\documentclass{amsart}

\usepackage{amsmath,amssymb,color}

\usepackage[latin1]{inputenc}
\makeatletter \@addtoreset{equation}{section} \makeatother

\renewcommand\thefigure{\thesection.\@arabic\c@figure}
\renewcommand\thetable{\thesection.\@arabic\c@table}
\newtheorem{theorem}{Theorem}[section]
\newtheorem{lemma}[theorem]{Lemma}
\newtheorem{proposition}[theorem]{Proposition}

\newtheorem{remark}[theorem]{Remark}

\newcommand{\mc}[1]{{\mathcal #1}}
\newcommand{\mf}[1]{{\mathfrak #1}}

\newcommand{\bb}[1]{{\mathbb #1}}

\newcommand{\<}{\langle}
\renewcommand{\>}{\rangle}
\newcommand{\x}{\!\otimes}

\begin{document}

\title[Equilibrium fluctuations]{Equilibrium fluctuations for gradient exclusion processes with conductances in random environments}

\author [J. Farfan] {Jonathan Farfan}
\author [A.B. Simas] {Alexandre B. Simas}
\author [F.J. Valentim]{Fábio J. Valentim}
\address{\hfill\break\indent
IMPA \hfill\break\indent
Estrada Dona Castorina 110, \hfill\break\indent
J. Botanico, 22460 Rio
de Janeiro, Brazil}

\thanks{Research supported by CNPq}

\noindent


\begin{abstract}

We study the equilibrium fluctuations for a gradient exclusion process with conductances in random environments, which can be viewed as a central limit theorem for the empirical distribution of particles
 when the system starts from an equilibrium measure.
\end{abstract}
\maketitle
\section{Introduction}
In this article we study the equilibrium fluctuations for a gradient exclusion process with conductances in random environments, which can be viewed as a central limit theorem for the empirical distribution of particles
 when the system starts from an equilibrium measure.

Let $W:\bb R^d\to \bb R$ be a function such that
$ W(x_1,\ldots,x_d) = \sum^d_{k=1}W_k(x_k)$, 
where $d\geq 1$ and each function $W_k: \bb R \to \bb R$ is strictly increasing, right continuous with left limits (c\`adl\`ag), and periodic in the sense that  $W_k(u+1) - W_k(u) = W_k(1) - W_k(0),$ for all $u\in \bb R$. The function $W$ will play the role of conductances in our system.
 
The random environment we considered is governed by the coefficients of the discrete formulation of the model (the process on the lattice). We will assume the underlying random field is ergodic, stationary and satisfies an ellipticity condition.

Informally, the exclusion process with conductances associated to $W$ is an interacting particle systems on the $d$-dimensional discrete torus $N^{-1}\bb T^d_N$, in which at most one particle per site is allowed, and only nearest-neighbor jumps are permitted. Moreover, the jump rate in the direction $e_j$ is proportional to the reciprocal of the increments of $W$ with respect to the $j$th coordinate.
Such a system can be understood as a model for diffusion in heterogeneous media. 

The purpose of this article is to study the density fluctuation field of this system as $N \to \infty$, and also the influence
of the randomness in this limit. For
any realization of the random environment, the scaling limit depends on the randomness only through some constants
which depend on the distribution of the random transition rates, but not on the particular realization of the random environment.

The evolution of one-dimensional exclusion processes with random conductances has attracted some attention recently \cite{jl,f,fjl,tc}. In all of these papers, a hydrodynamic limit was proved. The hydrodynamic limit may be interpreted as a law of large numbers for the empirical density of the system. Our goal is to go beyond the hydrodynamic limit and provide a new result for such processes, which is the equilibrium fluctuations and can be seen as a central limit theorem for the empirical density of the process.

To prove the equilibrium fluctuations, we would like to call attention to the main tools we needed: (i) the theory of nuclear spaces and (ii) homogenization of differential operators. The first one followed the classical approach of Kallianpur and Perez-Abreu \cite{kp} and Gel'fand and Vilenkin \cite{g}. Nuclear spaces are very suitable to attain existence and uniqueness of solutions for a general class of stochastic differential equations. Furthermore, tightness of processes on such spaces was established by Mitoma \cite{Mit}. A wide literature on these spaces can be found cited inside the fourth volume of the amazing collection by Gel'fand \cite{g}. The second tool is motivated by several applications in mechanics, physics, chemistry and engineering. We will consider stochastic homogenization. In the stochastic context, several works on homogenization of operators with random coefficients have been published (see, for instance, \cite{papa,pr} and references therein). In homogenization theory, only the stationarity of such random field is used. The notion of stationary random field is formulated in such a manner that it covers many objects of non-probabilistic nature, e.g., operators with periodic or quasi-periodic coefficients. We follow the approach given in \cite{sv}, which was introduced by \cite{pr}. 

The focus of our approach is to study the asymptotic behavior of effective coefficients for a family of random difference schemes, whose coefficients can be obtained by the discretization of random high-contrast lattice structures. Furthermore, the introduction of a corrected empirical measure was needed. The corrected empirical measure was used in the literature, for instance, by \cite{jl,tc,gj,v,sv}. It can be understood as a version of Tartar's  compensated compactness lemma in the context of particle systems. In this situation, the averaging due to the dynamics and the inhomogeneities introduced by the random media factorize after introducing the corrected empirical process, in such a way that we can average them separately. It is noteworthy that we managed to prove an equivalence between the asymptotic behavior with respect to both the corrected empirical measure and the uncorrected one. This equivalence was helpful in the sense that whenever the calculation with the corrected empirical measure turned cumbersome, we changed to a calculation with respect to the uncorrected one, and the other way around. This whole approach made the proof a lot more simpler than the usual one with respect solely to the corrected empirical measure developed in the articles mentioned above.

We now describe the organization of the article. In Section \ref{sec2} we state the main results of the article; in Section \ref{nuclear} we define the nuclear space needed in our context; in Section \ref{ef} we recall some results obtained in \cite{sv} about homogenization, and then we prove the equilibrium fluctuations by showing that the density fluctuation field converges to a process that solves the martingale problem. We also show that the solution of the martingale problem corresponds to a generalized Ornstein-Uhlenbeck process. In Section \ref{tig} we prove tightness of the density fluctuation field, as well as tightness of other related quantities. In Section \ref{s4} we prove the Boltzmann-Gibbs principle, which is a key result for proving the equilibrium fluctuations. Finally, the Appendix contains some known results about nuclear spaces and stochastic differential equations evolving on topologic dual of such spaces.

\section{Notation and results}\label{sec2} 
Denote by $\bb T^d = ({\bb R}/{\bb Z})^d = [0, 1)^d$ the $d$-dimensional torus, and by
$\bb T^d_N=(\bb Z/N\bb Z)^d = \{0,\ldots,N-1\}^d$ the $d$-dimensional discrete torus with $N^d$
points.

Fix a function $W: \bb R^d \to \bb R$ such that
\begin{equation}
\label{w}
W(x_1,\ldots,x_d) = \sum^d_{k=1}W_k(x_k),
\end{equation}
where each $W_k: \bb R \to \bb R$ is a \emph{strictly increasing} right continuous function with left
limits (c\`adl\`ag),
 periodic in the sense that for all $u\in \bb R$
 $$W_k(u+1) - W_k(u) = W_k(1) - W_k(0).$$ 
 
Define the generalized derivative $\partial_{W_k}$ of a function $f:\bb T^d \to \bb R$ by
\begin{equation}
\label{f004}
\partial_{W_k} f (x_1,\!\ldots\!,x_k,\ldots, x_d) = \lim_{\epsilon\rightarrow 0}
\frac{f(x_1,\!\ldots\!,x_k +\epsilon,\ldots, x_d)
-f(x_1,\!\ldots\!,x_k,\!\ldots\!, x_d)}{W_k(x_k+\epsilon) -W_k(x_k)}\;,
\end{equation} 
when the above limit exists and is finite. If for a function $f:\bb T^d\to\bb R$ the generalized derivatives $\partial_{W_k}$ exist for all $k=1,\ldots,d$, denote the generalized gradient of $f$ by
$$\nabla_W f = \left(\partial_{W_1}f,\ldots,\partial_{W_d}f\right).$$
Further details on these generalized derivatives can be found in subsection \ref{lw} and in the article \cite{sv}.

We now introduce the statistically homogeneous rapidly oscillating coefficients that will be used to define the random rates of the exclusion process with conductances in which we want to study the equilibrium fluctuations.

Let $(\Omega,\mc F, \mu)$ be a standard probability space and $\{ T_x : \Omega \to \Omega; x\in \bb Z^d\}$
be a group of $\mc F$-measurable and ergodic transformations which preserve the measure $\mu$:
\begin{itemize}
\item $ T_x : \Omega \to \Omega$ is $\mc F$-measurable for all $x \in \bb Z^d$,
\item $\mu(T_x \textbf{A}) = \mu(\textbf{A})$, for any $\textbf{A} \in \mc F$ and $x\in \bb Z^d$,
\item $T_0 = \textit{I}\;,\; \;T_x\circ T_y = T_{x+y}$,
\item For any  $f\in L^1(\Omega)$ such that $f(T_x\omega)=f(\omega)\;\; \mu$-a.s  for each $x\in \bb Z^d$, 
is equal to a constant $\mu$-a.s.
\end{itemize}
The last condition implies that the group $T_x$ is ergodic. 

Let us now introduce the vector-valued $\mc F$-measurable functions $\{a_j(\omega) ; j=1,\ldots, d\}$ such that there exists
$\theta >0$ with $$\theta^{-1} \le a_j(w)\le \theta,$$ 
for all $\omega \in \Omega$ and $ j= 1,\ldots, d$.  Then, define the diagonal matrices $A^N$ whose elements  are given by
\begin{equation}
\label{AN}
a^N_{jj}(x):=a^N_j = a_j(T_{Nx}\omega)\;,\;\;x\in T^d_N\;,\;\;j = 1, \ldots , d.
\end{equation}

Fix a typical realization $\omega \in \Omega$ of the random environment. For each $x\in \bb T^d_N$ and $j = 1,\ldots, d$, 
define the symmetric rate $\xi_{x, x+e_j}=\xi_{x+e_j, x}$ by
\begin{equation}\label{rate}
\xi_{x, x+e_j} \;=\; \frac{a^N_j(x)}{N[W((x+e_j)/N) - W(x/N)]}\;=\;
\frac{a^N_j(x)}{N[W_j((x_j+1)/N) - W_j(x_j/N)]},
\end{equation}
where ${e_1,\ldots ,e_d}$ is the canonical basis of $\bb R^d$.

Distribute particles on $\bb T^d_N$ in such a way that each site
of $\bb T^d_N$ is occupied at most by one particle. Denote by $\eta$ the
configurations of the state space $\{0,1\}^{\bb T^d_N}$ so that $\eta(x) =0$
if site $x$ is vacant, and $\eta(x)=1$ if site $x$ is occupied.

The  exclusion process with conductances in a random environment is a continuous-time Markov process
$\{\eta_t : t\ge 0\}$ with state space $\{0,1\}^{\bb T^d_N} =\{\eta:\bb T^d_N \to \{0,1\}\}$, 
whose generator $L_N$ acts on functions $f:
\{0,1\}^{\bb T^d_N} \to \bb R$ as
\begin{equation}
\label{g4}
(L_N f) (\eta) \;=\;\sum^d_{j=1} \sum_{x \in \bb T^d_N} \xi_{x,x+e_j}c_{x,x+e_j}(\eta)\, 
\{ f(\sigma^{x,x+e_j} \eta) - f(\eta) \} \;,
\end{equation}
where $\sigma^{x,x+e_j} \eta$ is the configuration obtained from $\eta$
by exchanging the variables $\eta(x)$ and $\eta(x+e_j)$:
\begin{equation}
\label{g5}
(\sigma^{x,x+e_j} \eta)(y) \;=\;
\begin{cases}
\eta (x+e_j) & \text{ if } y=x,\\
\eta (x) & \text{ if } y=x+e_j,\\
\eta (y) & \text{ otherwise},
\end{cases}
\end{equation}
and
\begin{equation*}
c_{x,x+e_j}(\eta) \;=\; 1 \;+\; b \{ \eta(x-e_j) + \eta(x+2\ e_j)\}\;,
\end{equation*}
with $b> -1/2\;$, and where all sums are modulo $N$.

We consider the Markov process  $\{\eta_t : t\ge 0\}$ on the configurations $\{0,1\}^{\bb T^d_N}$ 
associated to the generator $L_N$ in the diffusive scale, i.e., $L_N$ is speeded up by $N^2$. 

We now describe the stochastic evolution of the process. After a time given by an exponential distribution, 	
a random choice of a point $x\in \bb T^d_N$ is made. At rate $\xi_{x,x+e_j}$ the occupation variables
$\eta(x)$, $\eta(x+e_j)$ are exchanged. Note that only nearest neighbor jumps are allowed. The conductances are given by the function $W$, whereas the random environment is given by the matrix $A^N:=(a_{jj}^N(x))_{d\times d}$. The discontinuity points of $W$ may, for instance, model a membrane which obstructs the passage of particles in a fluid. For more details  see \cite{tc,v,sv}.

The effect of the factor $c_{x,x+e_j}(\eta)$ is the following: if the parameter $b$ is positive, the presence of particles in the neighboring
sites of the bond $\{x,x+e_j\}$ speeds up the exchange rate by a factor of
order one, and if the parameter $b$ is negative, the presence of particles in the neighboring sites slows down the exchange rate also by a factor of order one. More details are given in Remark \ref{taxac} below.

The dynamics informally presented describes a Markov evolution.
 A computation shows that the Bernoulli product measures
$\{\nu^N_\rho : 0\le \rho \le 1\}$ are invariant, in fact
reversible, for the dynamics. The measure $\nu^N_\rho$ is obtained
by placing a particle at each site, independently from the other
sites, with probability $\rho$. Thus, $\nu^N_\rho$ is a product
measure over $\{0,1\}^{\bb T^d_N}$ with marginals given by
\begin{equation*}
\nu^N_\rho \{\eta : \eta(x) =1\} \;=\; \rho
\end{equation*}
for $x$ in $\bb T^d_N$. 

Consider the random walk $\{X_t\}_{t\ge0}$ of a particle in $\bb T^d_N$ induced by the generator $L_N$ given as follows.
Let $\xi_{x,x+e_j}$ given by \eqref{rate}. If the particle is on a site $x\in \bb T^d_N$, it will jump to $x+e_j$ with rate $N^2\xi_{x,x+e_j}$. Furthermore, only nearest neighbor jumps are allowed.
The generator $\bb L_N$ of the random walk $\{X_t\}_{t\ge0}$ acts on functions $f:\bb T^d_N \to \bb R$ as
\begin{equation*}
\bb L_N f\left(\frac{x}{N}\right) \; =\; \sum^d_{j=1} \bb L_N^j f\left(\frac{x}{N}\right),
\end{equation*}
where,
\begin{equation*}
\bb L_N^j f\Big(\frac{x}{N}\Big) = N^2\Big\{\xi_{x,x+e_j} \Big[f\Big(\frac{x+e_j}{N}\Big) - f\Big(\frac{x}{N}\Big)\Big]  +
\xi_{x-e_j,x} \Big[f\Big(\frac{x-e_j}{N}\Big) - f\Big(\frac{x}{N}\Big)\Big]\Big\}
\end{equation*}
It is not difficult to see that the following equality holds: 
\begin{equation}
\label{opdisc}
\bb L_N f(x/N) = \sum^d_{j=1}\partial^N_{x_j}(a^N_j\partial^N_{W_j}f)(x)\;:=\;\nabla^NA^N\nabla^N_Wf(x),
\end{equation}
where, $\partial^N_{x_j}$ is the standard difference operator:
\begin{equation*}
\partial^N_{x_j}f\left(\frac xN\right)\; =\; N\left[f\left(\frac{x+e_j}{N}\right)-f\left(\frac xN\right)\right]\;,
\end{equation*}
and $\partial^N_{W_j}$ is the $W_j$-difference operator:
\begin{equation*}
\partial^N_{W_j}f\left(\frac xN\right)\;=\; \frac{f\left(\frac{x+e_j}{N}\right) -f\left(\frac xN\right)}{W\left(\frac{x+e_j}{N}\right) -W\left(\frac xN\right)},
\end{equation*}
for $x\in\bb T^d_N$.
Several properties of the above operator have been obtained in \cite{sv}.  

 The counting measure $m_N$ on $N^{-1} \bb T^d_N$ is
reversible for this process.
This random walk plays an important role in the proof of the equilibrium fluctuations of the process $\eta_t$, as we will see in subsection \ref{martprob}. 

Now we state a central limit theorem for the empirical measure, starting from an equilibrium measure $\nu_\rho$. Fix $\rho>0$ and denote by $\mc S_W(\bb T^d)$ the generalized Schwartz space on  $\bb T^d$, whose definition as well as some properties are given in Section \ref{nuclear}.

Denote by ${Y}_{\cdot}^{N}$ the \emph{density fluctuation field}, which is the bounded linear functional acting on functions $G \in \mc S_W(\bb T^d)$ as
\begin{equation}\label{densityfield}
Y_t^N(G) = \frac{1}{N^{d/2}} \sum_{x \in \bb T_N^d} G(x) [\eta_t(x) -\rho].
\end{equation}

Let $D([0,T], X)$ be the path space of
c\`adl\`ag trajectories with values in a metric space $X$. 
In this way we have defined a process in $D([0,T], \mc S_W'(\bb T^d))$, where $\mc S_W'(\bb T^d)$ 
is the topologic dual of the space $\mc S_W(\bb T^d)$.

\begin{theorem}
\label{tfeq} Consider the fluctuation field $Y_{\cdot}^N$ defined above. 
Then, $Y_{\cdot}^N$ converges weakly to the unique $\mc S_W'(\bb T^d)$-solution, $Y_t \in D([0,T],S_W'(\bb T^d))$, of the stochastic differential equation
\begin{equation}\label{spde-1}
dY_t = \phi'(\rho)\nabla A\nabla_W Y_t dt + \sqrt{2\chi(\rho)\phi'(\rho)A} dN_t,
\end{equation}
where $\chi(\rho)=\rho(1-\rho)$, $\phi(\rho)=\rho+b\rho^2$, and $\phi'$ is the derivative of $\phi$, $\phi'(\rho)=1+2b\rho$, and $N_t$ is a $\mc S_W'(\bb T^d)$-valued mean-zero martingale, with quadratic variation
$$\<N(G)\>_t = t\int_{\bb T^d} \left[\partial_{W_j}G(x)\right]^2 d(x^j\otimes W_j).$$
Furthermore, $N_t$ is a Gaussian process with independent increments. More precisely, for each $G\in S_W(\bb T^d)$, $N_t(G)$ is a time deformation of a standard Brownian motion. The process 
$Y_t$ is known in the literature as the generalized Ornstein-Uhlenbeck process with characteristics $\phi'(\rho)\nabla A\nabla_W$ and $\sqrt{2\chi(\rho)\phi'(\rho)A}\nabla_W$.
\end{theorem}

The proof of this theorem is given in Section \ref{ef}.

\begin{remark}\label{taxac}
The specific form of the rates $c_{x,x+e_i}$ is not important, but two
conditions must be fulfilled. The rates must be strictly positive,
they may not depend on the occupation variables $\eta(x)$, $\eta(x+e_i)$,
but they have to be chosen in such a way that the resulting process is
\emph{gradient}. (cf. Chapter 7 in \cite{kl} for the definition of
gradient processes).

We may define rates $c_{x,x+e_i}$ to obtain any polynomial $\phi$ of the
form $\phi(\alpha) = \alpha + \sum_{2\le j\le m} a_j \alpha^j$, $m\ge
1$, with $1+ \sum_{2\le j\le m} j a_j >0$. Let, for instance, $m=3$.
Then the rates
\begin{align*}
\hat c_{x,x+e_i} (\eta)\;\;& =\;\; c_{x,x+e_i} (\eta)\;\;+\\
&b\left\{ \eta(x-2e_i) \eta(x-e_i) + \eta(x-e_i) \eta(x+2e_i) + \eta(x+2e_i)
\eta(x+3e_i)\right\},
\end{align*}
satisfy the above three conditions, 
where $c_{x,x+e_i}$ is the rate defined at the beginning of Section 2 and
$a$, $b$ are such that $1+2a + 3b>0$. An elementary computation shows
that  
$\phi(\alpha) = 1 + a \alpha^2 + b \alpha^3$.
\end{remark}
\section{The space $\mc S_W(\bb T^d)$}\label{nuclear}
In this Section we build the space $\mc S_W(\bb T^d)$, which is associated to the operator $\mc L_W = \nabla \nabla_W$. This space, as we shall see, is a natural environment to attain existence and uniqueness of solutions of the stochastic differential equation \eqref{spde-1}. Furthermore, several lemmas are obtained to fulfill the conditions to ensure existence and uniqueness of such solutions.

\subsection{The operator $\mc L_W$}\label{lw}

Consider the operator $\mc L_{W_k}: \mc D_{W_k}\subset L^2(\bb T) \rightarrow \bb R$ given by
\begin{equation}
\label{f008}
\mc L_{W_k}f\;=\;  \partial_{x_k} \, \partial_{W_k} \, f,\;
\end{equation}
whose domain $\mc D_{W_k}$ is completely characterized in the following proposition:
\begin{proposition}
\label{dominiodw}
The domain $\mc D_{W_k}$ consists of all functions $f$ in $L^2(\bb T)$ such
that
\begin{equation*}
f(x) \;=\; a \;+\; b W_k(x) \;+\; \int_{(0,x]} W_k(dy) \int_0^y \mf f(z) \, dz
\end{equation*}
for some function $\mf f$ in $L^2(\bb T)$ that satisfies
\begin{equation*}
\int_0^1 \mf f(z) \, dz \;=\; 0\quad\hbox{~and~} \quad
\int_{(0,1]} W_k(dy) \Big\{ b + \int_0^y \mf f(z) \, dz \Big\} \;=\;0\; .
\end{equation*}
\end{proposition}
The proof of Proposition \ref{dominiodw} and further details can be found in \cite{tc}.  Furthermore, they also proved that these operators have a  countable complete orthonormal system of eigenvectors, which we denote by $\mc A_{W_k}$. Then, following \cite{v},
$$\mc A_W\;=\;\{f: \bb T^d \rightarrow \bb R;f(x_1,\ldots , x_d)=\prod^{d}_{k=1}{ f_k(x_k)}, f_k \in  \mc A_{W_k } \},$$
where $W$ is given by \eqref{w}. 

We may now build an operator analogous to $\mc L_{W_k}$ in $\bb T^d$. For a given set $\mc A$, we denote by $span(\mc A)$ the linear space generated by $\mc A$. Let $\bb D_W = span(\mc A_W)$, and define the operator $\bb L_W:\bb D_W \to L^2(\bb T^d)$ as follows: for $f =\prod^d_{k=1}{f_k}\in \mc A_W$,
\begin{equation}
\label{eq31}
\bb L_W(f)(x_1,\ldots x_d)=\sum^{d}_{k =1} \prod^{d}_{j=1, j\neq k}
{f_j(x_j)}\mc L_{W_k}f_k(x_k),
\end{equation}
and extend to $\bb D_W$ by linearity. It is easy to see that if $f\in \bb D_W$,
\begin{equation}
\label{f002}
\bb L_W f=\sum^{d}_{k =1}\mc L_{W_k}f,
\end{equation}
where the application of $\mc L_{W_k}$ on a function $f:\bb T^d\to\bb R$ is the natural one, i.e., it considers $f$ only as a function of the $k$th coordinate, and keeps all the remaining coordinates fixed. 

Let, for each $k=1,\ldots,d$, $f_k\in \mc A_{W_k}$ be an eigenvector
of $\mc L_{W_k}$ associated to the eigenvalue $\lambda _k$. Then $f=\prod^d_{k=1}{f_k}$ belongs to $\bb D_W$ and is an eigenvector of
 $\bb L_W$ with eigenvalue $\sum^d_{k=1}{\lambda _k}$. Moreover, \cite{v} proved the following result: 
\begin{lemma}
\label{f17}
The following statements hold:

\renewcommand{\theenumi}{\alph{enumi}}
\renewcommand{\labelenumi}{{\rm (\theenumi)}}

\begin{enumerate}
\item The set $\bb D_W$ is dense in $L^2(\bb T^d)$;
\item The operator $\bb L_W : \bb D_W \to L^2(\bb T^d)$ is symmetric and
  non-positive:
\begin{eqnarray*}
\< -\bb L_W f , f\> \;\ge\; 0,
\end{eqnarray*}
where $\<\cdot,\cdot\>$ is the standard inner product in $L^2(\bb T^d)$.
\end{enumerate}

\end{lemma}
Also, the set $\mc A_W$ forms a complete, orthonormal, countable system
of eigenvectors for the operator $\bb L_W$. Let
$\mc A_W= \{ \varphi_j\}_{ j\ge 1}$, $\{ \alpha_j\}_{ j\ge 1} $ be the corresponding eigenvalues of $-\bb L_W$, 
and consider $\mc D_W = \{v=\sum^{\infty}_{j=1}{v_j\varphi_j}\in L^2(\bb T^d);
\sum^{\infty}_{j=1}{v^2_j\alpha^2_j}<+\infty \}$. We
define the operator $\mc L_W:\mc D_W\to L^2(\bb T^d)$ by
\begin{equation}
-\mc L_Wv = \sum^{+\infty}_{j=1}{\alpha_{j}v_{j}\varphi_{j}}
\end{equation}

The operator $\mc L_W$ is clearly an extension of the operator $\bb L_W$,
and we present in Proposition \ref{t01} some properties of this operator.

\begin{proposition}
\label{t01}
The operator $\mc L_W : \mc D_W \to L^2(\bb T^d)$ enjoys the
following properties:

\renewcommand{\theenumi}{\alph{enumi}}
\renewcommand{\labelenumi}{{\rm (\theenumi)}}

\begin{enumerate}
\item  The domain $\mc D_W$ is
 dense in $L^2(\bb T^d)$. In particular, the set of eigenvectors
$\mc A_W=\{\varphi_j\}_{j\ge 1}$ forms a complete orthonormal system;
\item The eigenvalues of the operator $- \mc L_W$ form a countable set
  $\{\alpha_j\}_{ j\ge 1}$. 
  All eigenvalues have finite multiplicity,
 and it is   possible to obtain a re-enumeration $\{\alpha_j\}_{j\ge 1}$ 
 such that
$$0= \alpha_1 \le \alpha_2 \le \cdots \;\;\;\text{and}\;\; \lim_{n\to\infty} \alpha_n  = \infty;$$
\item The operator $\bb I - \mc L_W : \mc D_W \to L^2(\bb T^d)$ is  bijective;
\item $\mc L_W: \mc D_W \to L^2(\bb T^d)$ is self-adjoint and
  non-positive:
\begin{eqnarray*}
\< -\mc L_W f , f\> \;\ge\; 0;
\end{eqnarray*}

\item $\mc L_W$ is dissipative.

\end{enumerate}
\end{proposition}

\subsection{Nuclear spaces $\mc S_W(\bb T^d)$}
Our goal is to build a Countably Hilbert Nuclear space associated the self-adjoint operator $\mc L_W$. The reader is referred to Appendix.

Let $\{\varphi_j\}_{j\ge1}$ be the complete orthonormal set of the eigenvectors (in $L^2(\bb T^d)$) of the operator $\mc L = \bb I -\mc L_W$, and $\{\lambda_j\}_{j\ge1}$ the associated eigenvalues. Note that $\lambda_j = 1+\alpha_j$.

Consider the following increasing sequence $\|\cdot \|_n$, $n\in \bb N$, of Hilbertian norms:
$$\<f,g\>_n = \sum_{k=1}^\infty\<\bb P_kf, \bb P_kg\>\lambda^{2n}_k k^{2n},$$
where we denote by $\bb P_k$ the orthogonal projection on the linear space generated by the eigenvector $\varphi_k$.

So,

$$\|f\|_n^2 = \sum_{k=1}^\infty\|\bb P_kf\|^2\lambda^{2n}_k k^{2n},$$
where $\|\cdot\|$ is the $L^2(\bb T^d)$ norm.

Consider the Hilbert spaces $\mc S_n$ which are obtained by completing the space $\bb D_W$ with respect to the inner product $\<\cdot,\cdot\>_n$.

The set 
$$\mc S_W(\bb T^d) = \bigcap_{n=0}^\infty\mc S_n$$
endowed with the metric \eqref{metrica} is our \textit{countably Hilbert space}, and even more, it 
is a countably Hilbert nuclear space, see the Appendix for further details. In fact, fix $n\in \bb N$, and let $m >n+1/2$. We have that $\{\frac1{{(j\lambda_j)}^{m}}\varphi_j\}_{j\ge1}$ is a complete orthonormal set in $\mc S_m$. Therefore,
$$\sum^\infty_{j=1}\|\frac1{{(j\lambda_j)}^{m}}\varphi_j\|_n^2 \le \sum^\infty_{j=1}\frac {1}{j^{2(m-n)}}\; < \infty.$$

\begin{lemma}\label{semi}
Let $\mc L_W: \mc D_W\to L^2(\bb T^d)$ be the operator obtained in theorem \ref{t01}. We have 
\renewcommand{\theenumi}{\alph{enumi}}
\renewcommand{\labelenumi}{{\rm (\theenumi)}}

\begin{enumerate}
\item $\mc L_W$ is the generator of a strongly continuous contraction semigroup $\{P_t:L^2(\bb T^d) \to L^2(\bb T^d)\}_{t\ge0}$;
\item $\mc L_W$  is a closed operator;
\item For each $f\in L^2(\bb T^d)$, $t\mapsto P_tf$ is a continuous function from $[0,\infty)$ to $L^2(\bb T^d)$;
\item $\mc L_WP_tf\;=\;P_t\mc L_Wf$ for each  $f\in \mc L_W$ and $t\ge0$;
\item $(\bb I - \mc L_W)^nP_tf\;=\;P_t(\bb I - \mc L_W)^nf$ for each  $f\in \bb D_W$,$t \ge0$ and $n\in\bb N$;
\end{enumerate}
\end{lemma}
\begin{proof}
In view of (a), (b) and (d) in Theorem \ref{t01}, we may use Hille-Yosida Theorem to conclude the item (a) of the Lemma.
By item (a) proved, is easy to conclude items (b), (c) and (d) see, for instance, \cite[chapter 1]{ek}.
Considering  that $\mc L_Wf = \bb L_Wf$ for  $f\in \bb D_W$, and the application successive of the operator $\mc L_W$ is permited in this space. By item (d) follows item (e).
\end{proof}

The next Lemma permits conclude that the semigroup $\{P_t:t\ge0\}$ acting on the domain $\mc S_W(\bb T^d)$  is a $C_{0,1}$-semigroup.
\begin{lemma}\label{c01}
Let $\{P_t:t\ge0\}$ the semigroup whose infinitesimal generator is $\mc L_W$. Then for each $q\in \bb N$ we have:
\begin{equation*}
\|P_tf\|_q\le \|f\|_{q},
\end{equation*}
for all $f\in\mc S_W(\bb T^d)$. In particular, $\{P_t:t\ge0\}$ is a $C_{0,1}$-semigroup.
\end{lemma}
\begin{proof}
Let $f\in \bb D_W$, then
$$f = \sum_{j=1}^k \beta_j \varphi_{j},$$
for some $k\in\bb N$, and some constants $\beta_1,\ldots,\beta_k$. Using Hille-Yosida's theorem, a simple calculation shows that
$$P_t f = \sum_{j=1}^k \beta_j e^{t(1-\lambda_j)}\varphi_{j}.$$
Therefore, for $f\in\bb D_W$:
\begin{eqnarray*}
\|P_t f\|_n^2 &=& \|\sum_{j=1}^k\beta_je^{t(1-\lambda_j)}\varphi_{j}\|_n\\
&=&\sum_{j=1}^k\|\beta_j e^{t(1-\lambda_j)}\varphi_{j}\|^2 \lambda_j^{2n} j^{2n}\\
&\leq& \sum_{j=1}^k \|\beta_j \varphi_{j}\|^2 \lambda_j^{2n} j^{2n} = \|f\|_n^2
\end{eqnarray*}
We conclude the lemma by using the density of $\bb D_W$ in $\mc S_W(\bb T^d)$.
\end{proof}

\begin{lemma}\label{continua}
The operator $\mc L_W$ belongs to $\mc L(\mc S_W(\bb T^d),\mc S_W(\bb T^d))$, the space of linear continuous operators from $\mc S_W(\bb T^d)$ into $\mc S_W(\bb T^d)$.
\end{lemma}
\begin{proof}
Let $f\in \mc S_W(\bb T^d)$, and $\{\varphi_j\}_{j\ge 1}$ be the complete orthonormal set of eigenvectors of $\mc L_W$, with $\{(1-\lambda_j)\}_{j\ge 1}$ being their respectively eigenvalues. We have that
$$f = \sum_{j=1}^\infty \beta_j \varphi_j,\quad\hbox{with}\quad\sum_{j=1}^\infty \beta_j^2 <\infty.$$
We also have that
$$\mc L_W f = \sum_{j=1}^\infty (1-\lambda_j)\beta_j\varphi_j.$$
For every $n\in\bb N$:
\begin{eqnarray*}
\|\mc L_Wf\|_n^2 &=&\sum_{k=1}^\infty \|\bb P_k(\mc L_W f)\|^2 \lambda_k^{2n} k^{2n}= \sum_{k=1}^\infty \|\beta_k(1-\lambda_k)\varphi_k\|^2 \lambda_k^{2n} k^{2n}\\
&=& \sum_{k=1}^\infty \|\beta_k\varphi_k\|^2 (1-\lambda_k)^2 \lambda_k^{2n} k^{2n}\\
&\leq& 2\sum_{k=1}^\infty \|\bb P_k f\|^2 \lambda_k^{2n}k^{2n} + 2\sum_{k=1}^\infty \|\bb P_k f\|^2\lambda_k^{2(n+1)}k^{2(n+1)}\\
&=&2(\|f\|_{n} + \|f\|_{n+1}).
\end{eqnarray*}

Therefore, by the definition of $\mc S_W(\bb T^d)$, $\mc L_W f$ belongs to $\mc S_W(\bb T^d)$. Furthermore, $\mc L_W$ is continuous from $\mc S_W(\bb T^d)$ to $\mc S_W(\bb T^d)$.
\end{proof}

\section{Equilibrium Fluctuations}\label{ef}
We begin by stating some results on homogenization of differential operators obtained in \cite{sv}, which will be very useful along this Section. 

Let $L^2_{x^i\x W_i}({\bb T}^d)$ be the space of square integrable functions with respect to the product measure $d x_1\otimes\cdots\otimes d x_{i-1}\otimes d W_i \otimes d x_{i+1}\otimes\cdots\otimes d x_d$, and $H_{1,W}(\bb T^d)$ be the $W$-Sobolev space introduced in \cite{sv}.

Let $\lambda>0$, $f$ be a functional on $H_{1,W}(\bb T^d)$, $u_N$ be the unique weak solution 
of $$\lambda u_N - \nabla^N A^N \nabla_W^N u_N = f,$$ 
and $u_0$ be the unique weak solution of
\begin{equation}\label{h4}
\lambda u_0 - \nabla A\nabla_W u_0 = f.
\end{equation}
For more details on existence and uniqueness of such solutions see \cite{sv}.

In this context, we say that the diagonal matrix $A$ is a \textit{homogenization} of the sequence of random matrices $A^N$, denoted by $A^N\stackrel{H}{\longrightarrow} A$, if the following conditions hold:
\begin{itemize}
\item $u_N$ converges weakly in $ H_{1,W}(\bb T^d)$ to $u_0$, when $N\to\infty$;\\
\item $a_i^N \partial_{W_i}^N u^N \to a_i \partial_{W_i} u,$ weakly in $L^2_{x^i\x W_i}({\bb T}^d)$ when $N\to\infty$.
\end{itemize}
\begin{theorem}\label{homoge}
Let $A^N$ be a sequence of ergodic random matrices, such as the one that defines our random environment. Then, almost surely, $A^N(\omega)$ admits a homogenization, where the homogenized matrix $A$ does not depend on the realization $\omega$.
\end{theorem}
The following proposition regards the convergence of energies:
\begin{proposition}\label{hconvergencia}
Let $A^N\stackrel{H}{\longrightarrow} A$, as $N\to\infty$, with $u_N$ being the solution of
$$\lambda u_N - \nabla^N A^N \nabla_W^N u_N = f,$$
where $f$ is a fixed functional on $H_{1,W}(\bb T^d)$. Then, the following limit relations hold true:
$$\frac{1}{N^d}\sum_{x\in\bb T_N^d} u_N^2(x) \to \int_{\bb T^d} u_0^2(x) dx,$$
and
\begin{align*}
\frac{1}{N^{d-1}}\sum_{j=1}^d\sum_{x\in\bb T_N^d} a_{jj}^N(x) (\partial_{W_j}^N u_N(x))^2 &\left[W_j((x_i+1)/N)-W_j(x_i/N)\right]\\
 &\to \sum_{j=1}^d \int_{\bb T^d} a_{jj}(x) (\partial_{W_j} u_0(x))^2 d(x^j\x W_j),
\end{align*}
as $N\to\infty$.
\end{proposition}
\subsection{Martingale Problem}\label{martprob}

Recall that $Y^N_.$ is the bounded linear functional acting on functions $G\in \mc S_W(\bb T^d)$:
\begin{equation}\label{flutdiscr}
Y^N_t(G)  = \frac1{N^{d/2}}\sum_{x\in \bb T^d}G(x)\big[\eta_t(x)-\rho\big].
\end{equation}
This process $Y^N_.$ is called \textit{density fluctuation field}.

Denote by $Q_N$ the distribution in $D([0, T ], \mc S_W(\bb T^d))$ induced by the process $Y^N_t$ 
and initial distribution $\nu_\rho$. In order to prove the martingale problem we introduce
the \textit{corrected density fluctuation field} defined on  solutions functions of the a appropriate problem of homogenization:
\begin{equation}\label{flutdiscrcor}
Y^{N,\lambda}_t(G)  = \frac1{N^{d/2}}\sum_{x\in \bb T^d}G^\lambda_N(x)\big[\eta_t(x)-\rho\big],
\end{equation}
where $G_N^\lambda$ is the weak solution for the equation
\begin{equation}\label{prob homo}
\lambda G_N^\lambda - L_NG_N^\lambda = \lambda G - \nabla A \nabla_W G 
\end{equation}
that, via homogenization, converges to $G$ which is the trivial solution of the problem
\begin{equation*}
\lambda G - \nabla A \nabla_W G = \lambda G - \nabla A \nabla_W G. 
\end{equation*}

The processes $Y_\cdot^N$ and $Y_\cdot^{N,\lambda}$ have the same asymptotic behavior, as we will see. But some calculations are simpler with one of them than with the other. In this way, we have defined two processes in $D([0, T ],\mc S_W'(\bb T^d))$, where $\mc S_W'(\bb T^d)$ is the  topologic dual of the space $\mc S_W(\bb T^d)$.

Fix some process $Y_\cdot$ in $D([0, T ],\mc S_W'(\bb T^d))$, and for $t\ge0$, let $\mc F_t$ be the $\sigma$-algebra generated by $Y_s(H)$ for
$s\le t$ and $H\in \mc S_W(\bb T^d)$. Furthermore, set $\mc F_\infty = \sigma \Big(\bigcup_{t\ge 0}\mc F_t\Big)$. Denote by $Q^{\lambda}_N$
the distribution on $D([0, T ],\mc S_W'(\bb T^d))$ induced by the corrected density fluctuation field $Y^{N,\lambda}_\cdot$ and initial distribution $\nu_\rho$.

Theorem \ref{tfeq} is a consequence of the following result about the corrected fluctuation field.

\begin{theorem}\label{4.2}
Let $Q$ be the probability measure on $D([0,T],\mc S_W'(\bb T^d))$ corresponding to the generalized Ornstein-Uhlenbeck process of
mean zero and characteristics $\phi'(\rho)\nabla \cdot {A} \nabla_W$, $\sqrt{2\chi(\rho)\phi'(\rho){A}}\nabla_W$. Then the sequence $\{ Q_{N}^{\lambda}\}_{N\geq{1}}$ converges weakly to the probability measure $ Q$. \label{th:flu}
\end{theorem}
The next Lemma shows that tightness of $Y^{N,\lambda}_t$ follows from tightness of $Y^{N}_t$, and even more, that they have the same limit points. So we can conclude our main theorem from the Theorem \ref{4.2}
\begin{lemma}\label{lproc}
for all $t\in [0,T]$ and $G\in \mc S_W(\bb T^d)$,
$\lim_{N\to\infty}E_{\nu_{\rho}}\big[Y^N_t(G)- Y^{N,\lambda}_t(G)\big]^2 = 0.$
\end{lemma}
\begin{proof}
By convergence of energies, we have that $\lim_{N\to\infty}G^\lambda_N = G$ in $L^2_N(\bb T^d)$, i.e. 
\begin{equation}\label{f1}
\|G^\lambda_N - G\|_N^2 := \frac{1}{N^d}\sum_{x\in\bb T_N^d} [G_N^\lambda(x/N)-G(x/N)]^2\to 0,\qquad \text{as}\qquad\; N\to \infty. 
\end{equation}
Since $\nu_\rho$ is a product measure we obtain
$$E_{\nu_\rho}\big[Y^N_t(G)- Y^{N,\lambda}_t(G)\big]^2\;=\;$$
$$=\;E_{\nu_\rho}\big[\frac 1 {N^d}\sum_{x,y\in\bb T_N^d}[G^\lambda_N(x/N) - G(x/N)][G^\lambda_N(y/N) - G(y/N)](\eta_t(x)-\rho)(\eta_t(y)-\rho)\big]\;=$$

$$=\;E_{\nu_\rho}\big[\frac 1 {N^d}\sum_{x\in\bb T_N^d}[G^\lambda_N(x/N) - G(x/N)]^2(\eta_t(x)-\rho)^2\big]
\;\le \frac{C(\rho)}{N^d}\sum_{x\in\bb T_N^d}[G^\lambda_N(x/N) - G(x/N)]^2,$$
where $C(\rho)$ is a constant that depend on $\rho$. By \eqref{f1} the last expression vanishes as $N\to \infty$.
\end{proof}

\centerline{{\bf Proof of Theorem \ref{4.2}}}
Consider the martingale
\begin{equation}\label{mo}
M_t^{N}(G) = Y^N_t(G) - Y^N_0(G) - \int_0^t \, N^2 L_N Y^N_s(G)ds
\end{equation}
associated to the original process and
\begin{equation}\label{mc}
M_t^{N,\lambda}(G) = Y^{N,\lambda}_t(G) - Y^{N,\lambda}_0(G) - \int_0^t \, N^2 L_N Y^{N,\lambda}_s(G)ds
\end{equation}
associated to the corrected process.

A long, albeit simple, computation shows that the quadratic variation of the martingale $M^{N,\lambda}_t(G)$,
$\<M^{N,\lambda}(G)\>_t$, is given by:
\begin{align}\label{vq}
 \frac{1}{N^{d-1}}\sum^d_{j=1}\sum_{x\in \bb T^d}a_{jj}^N[\partial_{W_j}^N G_N^{\lambda}(x/N)]^2&[W((x+e_j)/N) - W(x/N)]\times\\
&\times\int_0^t c_{x,x+e_j}(\eta_s) \, [\eta_s(x+e_j) - \eta_s(x)]^2 \, ds\;.\nonumber
\end{align}

Is not difficult see that the quadratic variation of the martingale $M^{N}_t(G)$, $\<M^{N}(G)\>_t$, have the expression 
\eqref{vq} with $G$ replacing $G_N^{\lambda}$. Further,
\begin{equation}
E_{\nu_\rho}\big[c_{x,x+e_j}(\eta) \, [\eta_s(x+e_j) - \eta_s(x)]^2\big] = 2\chi(\rho)\phi'(\rho).
\end{equation}
\begin{lemma}\label{lvq}
Fix $G\in \mc S_W(\bb T^d)$ and $t>0$, and Let $\<M^{N,\lambda}(G)\>_t$ and $\<M^N(G)\>_t$ be the quadratic variation of the martingales $M^{N,\lambda}_t(G)$ and $M^N_t(G)$, 	
respectively. Then,
\begin{equation}\label{vq2}
\lim_{N\to\infty}E_{\nu_\rho}\big[\<M^{N,\lambda}(G)\>_t - \<M^N(G)\>_t\big]^2 \;=\;0.
\end{equation}
\begin{proof}
Fix $G\in \mc S_W(\bb T^d)$ and $t>0$. A straightforward calculation shows that
$$E_{\nu_\rho}\big[\<M^{N,\lambda}(G)\>_t - \<M^N(G)\>_t\big]^2 \le$$
$$\le\Big[k^2t^2\frac{1}{N^{d-1}}\sum^d_{j=1}\sum_{x\in \bb T^d}a_{jj}^N[\big(\partial_{W_j}^N G_N^{\lambda}(x/N)\big)^2-\big(\partial_{W_j}^N G(x/N)\big)^2][W((x+e_j)/N) - W(x/N)] \Big]^2.$$
Where the constant $k$ comes from the integral term. By the convergence of energies (Proposition \ref{hconvergencia}), the last term vanishes as $N\to \infty$.
\end{proof}
\end{lemma}

\begin{lemma}\label{f2} Let $G\in \mc S_W(\bb T^d)$ and $d>1$. Then
\begin{align*}
 \lim_{N\to \infty}E_{\nu_\rho}\Big[\frac{1}{N^{d-1}}\int_0^t \,ds&\sum^d_{j=1}\sum_{x\in \bb T^d}a_{jj}^N\big(\partial_{W_j}^N G(x/N)\big)^2[W((x+e_j)/N) - W(x/N)]\times\\
&\times \big[c_{x,x+e_j}(\eta_s) \, [\eta_s(x+e_j) - \eta_s(x)]^2 -2\chi(\rho)\phi'(\rho)\big]\Big]^2\;=\;0.
\end{align*}
\end{lemma}
\begin{proof}
Fix $G\in \mc S_W(\bb T^d)$. The term in previous expression is  less than or equal to 
\begin{equation*}
\frac{t^2\theta^2C(\rho)}{N^{d-1}}\|\nabla^N_WG^2\|_{W,N}\|\nabla^N_WG^2\|_{W,N},
\end{equation*}
where
$$\|\nabla^N_WG^2\|_{W,N}^2 := \frac{1}{N^{d-1}}\sum_{j=1}^d \sum_{x\in \bb T^d}\big(\partial_{W_j}^N G(x/N)\big)^2[W((x+e_j)/N) - W(x/N)].$$
For $d>1$, the previous term converges to zero as $N\to \infty$. 
\end{proof}
 The case $d=1$ follows from calculations similar to the ones found in \cite{mil}. So, by Lemma \ref{lvq} and \ref{f2}, $\<M^{N,\lambda}(G)\>_t$ is given by
$$ \frac{2t\chi(\rho)\phi'(\rho)}{N^{d-1}}\sum^d_{j=1}\sum_{x\in \bb T^d}a_{jj}^N\big(\partial_{W_j}^N G_N^\lambda(x/N)\big)^2[W((x+e_j)/N) - W(x/N)]$$
 plus a term that vanish in $L^2_{\nu_\rho}(\bb T^d)$ as $N\to \infty$. By the convergence of energies, Proposition \ref{hconvergencia}, it converges, as $N\to \infty$, to 
 $$ 2t\chi(\rho)\phi'(\rho)\sum^d_{j=1}\int_{\bb T^d}a_{jj}^N\big(\partial_{W_j} G(x)\big)^2dx^j\x W_j.$$

Our goal now consists in showing that it is possible to write the integral part of the martingale as the integral of a function of the density fluctuation field plus a term that goes to zero in $L^2_{\nu_\rho}(\bb T^d)$. 
By a long, but simple, computation, we obtain that
 
\begin{eqnarray*}
\!\!\!\!\!\!\!\!\!\!\!\!\!\! &&
N^2 L_N Y^{N,\lambda}_s(G)\;=\;\sum^d_{j=1}\big \{\frac {1}{N^{d/2}} \sum_{x\in  T^d_N}  {\bb L}_N^j G^\lambda_N(x/N)\, \eta_s(x)
\\
\!\!\!\!\!\!\!\!\!\!\!\!\!\! && \quad
+\; \frac{b}{N^{d/2}} \sum_{x\in  T^d_N} \big [  {\bb L}_N^j G^\lambda_N
((x+e_j)/N) +  {\bb L}_{N}^j G^\lambda_N (x/N) \big ] \,
(\tau_x h_{1,j}) (\eta_s) \\
\!\!\!\!\!\!\!\!\!\!\!\!\!\! && \qquad
- \; \frac{b}{N^{d/2}} \sum_{x\in  T^d_N}  {\bb L}_N^j G^\lambda_N
(x/N) (\tau_x h_{2,j}) (\eta_s)\big \}\;,
\end{eqnarray*}
where $\{\tau_x: x\in \bb Z^d\}$ is the group of translations, so that
$(\tau_x \eta)(y) = \eta(x+y)$ for $x$, $y$ in $\bb Z^d$, and the
sum is understood modulo $N$. Also, $h_{1,j}$, $h_{2,j}$ are the cylinder functions
\begin{equation*}
h_{1,j}(\eta) \;=\; \eta(0) \eta({e_j})\;,\quad h_{2,j}(\eta) \;=\; \eta(-e_j)  \eta(e_j)\;.
\end{equation*}

Note that inside the expression $N^2L_NY^{N,\lambda}_s$ we may replace ${\bb L}^j_NG^\lambda_N$ by $a_j\partial_{x_j}\partial_{W_j}G$. Indeed,

\begin{align*}
E_{\nu(\rho)}\Big \{\int_0^t\sum^d_{j=1}&\frac {1}{N^{d/2}} \sum_{x\in \bb T^d_N} \Big [ {\bb L}_N^j G^\lambda_N(x/N)-a_j\partial_{x_j} \partial_{W_j}G(x/N)\Big ]\, \big(\eta_s(x) - \rho\big)\;+
\\
+\; \frac{b}{N^{d/2}}& \sum_{x\in \bb T^d_N} \Big[  {\bb L}_N^j G^\lambda_N((x+e_j)/N) - a_j\partial_{x_j}\partial_{W_j}G((x+e_j)/N)\;+\\
& {\bb L}_N^j G^\lambda_N(x/N) - a_j\partial_{x_j} \partial_{W_j}G(x/N)\Big]\big((\tau_x h_{1,j})(\eta_s) - \rho^2\big)\;-\\
- \; &\frac{b}{N^{d/2}} \sum_{x\in \bb T^d_N}\Big[ {\bb L}_N^j G^\lambda_N(x/N) - a_j\partial_{x_j} \partial_{W_j}G(x/N)\Big] \big((\tau_x h_{2,j}) (\eta_s) - \rho^2\big)\Big \}^2.
\end{align*}

whereas the above expression is less than or equal to
\begin{equation*}
C(\rho,b)\int_0^t\frac 1{N^d}\sum_{x\in \bb T^d}\big[ L_NG^\lambda_N(x/N) - \nabla A\nabla_WG(x/N)\big]^2.
\end{equation*}
Now, recall that $G^\lambda_N$ is solution of the equation \eqref{prob homo}, and therefore, the previous expression is less than or equal to
\begin{equation*}
\frac{t\;C(\rho,b)}{\lambda^2}\|G^\lambda_N - G\|^2_N,
\end{equation*}
 thus, by homogenization and energy estimates in Theorem \ref{homoge} and Proposition \ref{hconvergencia}, respectively, the last expression converges to zero as $N\to \infty$.
 
By the Boltzmann Gibbs principle, Theorem \ref{th:bg},  we can replace $(\tau_x h_{i,j}) (\eta_s)- \rho^2$ by $2\rho[\eta_s(x) - \rho]$ for $ i=1,2$. 
Doing so, the martingale \eqref{mc} can be written as
\begin{equation}\label{ec8}
M_t^{N,\lambda}(G) = Y^{N,\lambda}_t(G) - Y^{N,\lambda}_0(G) - \int_0^t \, \frac1{N^{d/2}}\sum_{x\in\bb T^d}\nabla A\nabla_WG(x/N)\phi'(\rho)\big(\eta_s - \rho\big)ds,
\end{equation}
plus a term that vanishes in $L^2_{\nu_{\rho}}(\bb T^d)$ as $N \to \infty$.

Notice that, by \eqref{flutdiscr}, the integrand in the previous expression is a function of the density fluctuation field $Y^N_t$. By Lemma \ref{lproc}, we can replace the term inside the integral of the above expression by a term which is a function of the corrected density fluctuation field $Y^{N,\lambda}_t$.

From the results of Section \ref{tig}, the sequence $\{{Q}_{N}^{\lambda}\}_{N\geq{1}}$ is tight and let $Q^\lambda$ be a limit point of it. Let $Y_t$ be the process in $D([0,T], \mc S_W'(\bb T^d))$ induced by the canonical projections under $Q^\lambda$. Taking the limit as $N \to \infty$,
under an appropriate subsequence, in expression (\ref{ec8}), we obtain that
\begin{equation} \label{martingal}
M^{\lambda}_{t}(G)=Y_{t}(G)-Y_{0}(G) -\int^{t}_{0}{Y}_{s}(\phi'(\rho)\nabla \cdot {A}\nabla_W G)ds
\end{equation}
where $M^\lambda_t$ is some $\mc S_W'(\bb T^d)$-valued process. In fact, $M^\lambda_t$ is a martingale. To see this, note that for a measurable set $U$ with respect to the canonical $\sigma$-algebra $\mc F_t$, $E_{Q_N^{\lambda}}[M_t^{N,\lambda}(G)\mathbf 1_U]$ converges to $E_{Q^\lambda}[M_t^\lambda(G) \mathbf 1_U]$. Since $M_\cdot^{N,\lambda}(G)$ is a martingale, $E_{Q_N^\lambda}[M_T^{N,\lambda}(G) \mathbf 1_U] = E_{Q_N^\lambda}[M_t^{N,\lambda}(G) \mathbf 1_U]$. And taking a further subsequence if necessary, this last term converges to $E_{Q^\lambda}[M_t^\lambda(G) \mathbf 1_U]$, which proves that $M_\cdot^\lambda(G)$ is a martingale for any $G \in \mc S_W(\bb T^d)$. Since all the projections of $M^\lambda_t$ are martingales, we conclude that $M^\lambda_t$ is a $\mc S_W'(\bb T^d)$-valued Martingale.

Now, we need obtain the quadratic variation $\< M^\lambda(G)\>_t$ of the martingale $M_t^\lambda(G)$. A simple application of Tchebyshev's inequality proves that $\< M^{N,\lambda}(G)\>_t$ converges in probability to 
\[
2t\chi(\rho) \phi'(\rho) \sum_{j=1}^d\int_{\bb T^d} a_j\Big[\partial_{W_j}G\Big]^2d(x^j\x W_j),
\]
 Where $\chi(\rho)$ stand for the static compressibility given by $\chi(\rho) = \rho(1-\rho)$. 
Remember the definition of quadratic variation. We need to prove that 
$$M_t^\lambda(G)^2 - 2t\chi(\rho) \phi'(\rho) \sum_{j=1}^d\int_{\bb T^d} a_j\Big[\partial_{W_j}G\Big]^2d(x^j\x W_j)$$
is a martingale. The same argument we used above applies now if we can show that $\sup_N E_{Q_N^\lambda}[M_T^{N,\lambda}(G)^4]<\infty$ and $\sup_N E_{Q_N^\lambda}[\< M^{N,\lambda}(G)\>_T^2]<\infty$. Both bounds follows easily from the explicit form of $\<M^{N,\lambda}(G)\>_t$ and (\ref{ec8}).

On the other hand, by a standard central limit theorem, ${Y}_{0}$ is a Gaussian field with covariance 
$$E\big[Y_{0}(G) Y_{0}(H)\big]=\chi(\rho)\int_{\bb T^d}G(x)H(x)dx.$$
Therefore, by Theorem \ref{t5}, $Q^\lambda$ is equal to the probability distribution $Q$ of a generalized Ornstein-Uhlenbeck process in $D([0,T],\mc S'_W(\bb T^d))$ (and it does not depend on $\lambda$). By uniqueness of the generalized Ornstein-Uhlenbeck processes (also due to Theorem \ref{t5}), the sequence $\{Q_N^\lambda\}_{N\ge1}$ has at most one limit point, and from tightness, it does have a unique limit point. This concludes the proof of Theorem \ref{4.2}.

\subsection{Generalized Ornstein-Uhlenbeck Processes}
In this subsection we show that the generalized Ornstein-Uhlenbeck process obtained as the solution martingale problem which we are interested, is also a $\mc S_W'(\bb T^d)$-solution of a stochastic differential equation, and then we apply the theory in Appendix to conclude that there is at most one solution of the martingale problem. Moreover, we also conclude that this process is a Gaussian process.

\begin{theorem}
\label{t5}
 Let $Y_0$ be a Gaussian field on $\mc S_W'(\bb T^d)$, and let for any $G\in \mc S_W(\bb T^d)$
\begin{equation}\label{idintegral}
   M_t(G) = Y_t(G) - Y_0(G) - \phi'(\rho) \int_0^t  Y_s(\nabla A\nabla_W G) ds
\end{equation}
 be a martingale of quadratic variation 
 \begin{equation}\label{vq-m}
 \<M_t(G)\> =2t\chi(\rho)\phi'(\rho)\sum_{j=1}^d \int_{\bb T^d} a_{jj}\left(\partial_{W_j}G\right)^2 d(x^j\otimes W_j).
 \end{equation}
 Then $Y_t$ is the unique $\mc S_W'(\bb T^d)$-solution of the stochastic differential equation
\begin{equation}\label{ec7}
d Y_t = \phi'(\rho)\nabla A\nabla_W Y_t dt + \sqrt{2\chi(\rho)\phi'(\rho) A} d N_t,
\end{equation}
where $N_t$ is a mean-zero $\mc S_W'(\bb T^d)$-valued martingale with quadratic variation given by
$$\<N(G)\>_t = t\sum_{j=1}^d\int_{\bb T^d} \left[\partial_{W_j} G\right]^2 d(x^j\otimes W_j).$$ 
 Moreover, $Y_t$ is a Gaussian process.
\end{theorem}
\begin{proof}
In view of definition of solutions in Appendix, $Y_t$ is a  $\mc S_W'(\bb T^d)$-solution of \eqref{ec7}. In fact, by hypothesis $Y_t$ satisfies the integral identity \eqref{idintegral}, and is also an additive functional of a Markov process.

We now check the conditions in Proposition \ref{eu-spde} to ensure uniqueness of $\mc S_W'(\bb T^d)$-solutions of \eqref{ec7}.
Since by hypothesis $Y_0$ is a Gaussian field, condition 1 is satisfied, and since the martingale $M_t$ has the quadratic variation given by \eqref{vq-m}, we use Remark \ref{remark} to conclude that condition 2 holds. Condition 3 follows from Lemmas 
\ref{c01} and \ref{continua}.
Therefore $Y_t$ is unique.

Finally, by Blumenthal's 0-1 law for Markov processes, $M_t$ and $Y_0$ are independent. Applying Lévy's martingale characterization of Brownian motions, the quadratic variation of $M_t$, given by \eqref{vq-m}, yields that $M_t$ is a time deformation of a Brownian motion. Therefore, $M_t$ is a Gaussian process with independent increments. Since $Y_0$ is a Gaussian field, we apply Proposition \ref{gaussian} to conclude that $Y_t$ is a Gaussian process in $D([0,T],S_W'(\bb T^d))$.
\end{proof}

\section{Tightness}\label{tig}
In this section we prove tightness of the density fluctuation field $\{Y_\cdot^N\}_N$ introduced in Section \ref{sec2}. We begin by stating Mitoma's criterion \cite{Mit}:

\begin{proposition}
Let $\Phi_\infty$ be a nuclear Fr\'echet space and $\Phi_\infty'$ its topological dual. Let $\{Q^N\}_N$ be a sequence of distributions in $D([0,T],\Phi_\infty')$, and for a given function $G \in \Phi_\infty$, let $Q^{N,G}$ be the distribution in  $D([0,T],\bb R)$ defined by $Q^{N,G}\left[y\in D([0,T],\bb R); y(\cdot) \in A\right] = Q^{N}\left[Y\in D([0,T],\Phi_\infty'); Y(\cdot)(G) \in A\right]$. Therefore, the sequence $\{Q^N\}_N$ is tight if and only if  $\{Q^{N,G}\}_N$ is tight for any $G \in \Phi_\infty$.  
\end{proposition}

From Mitoma's criterion, $\{Y_\cdot^N\}_N$ is tight if and only if $\{Y_\cdot^N(G)\}_N$ is tight for any $G \in \mc S_W(\bb T^d)$, since $\mc S_W(\bb T^d)$ is a nuclear Fréchet space. By Dynkin's formula and after some manipulations, we see that
\begin{eqnarray}
\!\!\!\!\!\!\!\!\!\!\!\!\!\!Y_t^N(G) &=& Y_0^N(G)
\int_0^t \sum^d_{j=1}\big \{\frac {1}{N^{d/2}} \sum_{x\in \bb T^d_N} {\bb L}_N^j G_N(x/N)\, \eta_s(x)\nonumber
\\
\!\!\!\!\!\!\!\!\!\!\!\!\!\! && \quad
+\; \frac{b}{N^{d/2}} \sum_{x\in \bb T^d_N} \big [ {\bb L}_N^j G_N
((x+e_j)/N) + {\bb L}_{N}^j G_N (x/N) \big ] \,
(\tau_x h_{1,j}) (\eta_s) \nonumber\\
\!\!\!\!\!\!\!\!\!\!\!\!\!\! && \qquad
- \; \frac{b}{N^{d/2}} \sum_{x\in \bb T^d_N} {\bb L}_N^j G_N
(x/N) (\tau_x h_{2,j}) (\eta_s)\big \}ds\; + M_t^N(G),\label{eq08}
\end{eqnarray}
where $M_t^N(G)$ is a martingale of quadratic variation 
\begin{align*}\label{vq}
 \<M^N(G)\>_t = \frac{1}{N^{d-1}}\sum^d_{j=1}\sum_{x\in \bb T^d}a_{jj}^N&[\partial_{W_j}^N G_N(x/N)]^2[W((x+e_j)/N) - W(x/N)]\times\\
&\times\int_0^t c_{x,x+e_j}(\eta_s) \, [\eta_s(x+e_j) - \eta_s(x)]^2 \, ds\;.
\end{align*}

In order to prove tightness for the sequence $\{Y_\cdot^N(G)\}_N$, it is enough to prove tightness for $\{Y_0^N(G)\}_N$, $\{M_\cdot^N(G)\}_N$ and the integral term in (\ref{eq08}). The easiest one is the initial condition: from the usual central limit theorem, $Y_0^N(G)$ converges to a normal random variable of mean zero and variance $\chi(\rho) \int G(x)^2 dx$, where $\chi(\rho) = \rho(1-\rho)$. For the other two terms, we use {\em Aldous' criterion}:

\begin{proposition}[Aldous' criterion]
\label{p3}
A sequence of distributions $\{P^N\}$ in the path space $D([0,T],\bb R)$ is tight if:

\begin{itemize}
\item[i)] For any $t \in [0,T]$ the sequence $\{P_t^N\}$ of distributions in $\bb R$ defined by $P_t^N(A) = P^N\left[y \in D([0,T],\bb R); y(t) \in A\right]$ is tight,

\item[ii)] For any $\epsilon >0$,
\[
\lim_{\delta>0} \limsup_{n \to \infty} \sup_{\substack{\tau \in \Upsilon_T\\ \theta \leq \delta}} P^N\big[y\in D([0,T],\bb R);|y(\tau+\theta)-y(\tau)| >\epsilon\big] =0,
\]
\end{itemize}
where $\Upsilon_T$ is the set of stopping times bounded by $T$  and $y(\tau+\theta)=y(T)$ if $\tau+\theta>T$.
\end{proposition}

Now we prove tightness of the martingale term. By the optional sampling theorem, we have

\begin{align}
Q_N \big[ \big|&M_{\tau+\theta}^N(G)- M_\tau^N(G)\big|> \epsilon\big] 
	 \leq \frac{1}{\epsilon^2} E_{Q_N} \big[ \big\< M_{\tau+\theta}^N(G)\big\> - \big\< M_\tau^N(G)\big\>\big]\nonumber\\
	& =\frac{1}{\epsilon^2}\big[ \big\< M_{\tau+\theta}^N(G)\big\> - \big\< M_\tau^N(G)\big\>\big]\nonumber\\
	&=\frac{1}{\epsilon^2 N^{d-1}} \sum_{j=1}^d \sum_{x\in\bb T_N^d} a_{jj}(x) [\partial_{W_j}^N G(x/N)]^2[W((x+e_j)/N)-W(x)]\nonumber\\
	&\times \int_t^{t+\delta} c_{x,x+e_j}(\eta_s) [\eta_s(x+e_j) - \eta_s(x)]^2 ds\nonumber\\
	&\leq \frac{\delta}{\epsilon^2} (1+2|b|)\theta\frac{1}{N^{d-1}} \sum_{j=1}^d \sum_{x\in \bb T_N^d} [\partial_{W_j}^N G(x/N)]^2[W((x+e_j)/N)-W(x)]\label{ener1}\\
&\leq \frac{\delta}{\epsilon^2}(1+2|b|)\theta(\|\nabla_W G\|_W^2 +\delta),\nonumber
\end{align}
for $N$ sufficiently large, since the rightmost term on \eqref{ener1} converges to $\|\nabla_W G\|_W^2$, as $N\to\infty$.
Therefore, the martingale $M_t^N(G)$ satisfies the conditions of Aldous' criterion. The integral term can be handled in a similar way:

\begin{eqnarray*}
E_{Q_N}\Big[ &\Big(& \int_\tau^{\tau+\delta} \frac{1}{N^{d/2}} \sum_{j=1}^d \sum_x \Big\{ {\bb L}_N^jG(x/N)(\eta_t-\rho)\\
&+& b[{\bb L}_N^j G((x+e_j)/N) + {\bb L}_N^jG(x/N)](\tau_x h_1-\rho^2)\\
&-& b{\bb L}_N^j G(x/N) (\tau_x h_2 - \rho^2)\Big)^2dt\Big]\\
&\leq& \delta C(b) \frac{1}{N^d} \sum_{j=1}^d\sum_{x\in\bb T_N^d} \left({\bb L}_N^j G(x/N)\right)^2\\
&\leq& \delta C(G,b),
\end{eqnarray*}
where $C(b)$ is a constant that depends on $b$, and $C(G,b)$ is a constant that depends on $C(b)$ and on the function $G\in \mc S_W(\bb T^d)$. Therefore, we conclude, by Mitoma's criterion, that the sequence $\{Y_\cdot^N\}_N$ is tight. Thus, the sequence of $\mc S_W'(\bb T^d)$-valued martingales $\{M_\cdot^N\}_N$ is also tight.

\section{Boltzmann-Gibbs Principle}
\label{s4}
We show in this section that the martingales $M_t^N(G)$ introduced in Section \ref{ef} can be expressed in terms of the fluctuation fields $Y_t^N$. This replacement of the cylinder function $(\tau_x h_{i,j})(\eta_s) - \rho^2$ by $2\rho[\eta_s(x)-\rho]$ for $i=1,2$, constitutes one of the main steps toward the proof of equilibrium fluctuations.

Recall that $(\Omega, \mc F, \mu)$ is a standard probability space which we consider defined the vector-valued $\mc F$-measurable functions $\{a_j(\omega); j =\ldots, d\}$, these functions form our random environment (see Sections \ref{sec2} and \ref{ef} for more details), and $\{0,1\}^{\bb T^d_N}$ the space of configurations  on $\bb T^d_N$, ie, $\{\eta; \eta(x)\in \{0,1\},\ x\in \bb T^d_N\}$.

Take a function $f: \Omega \times \{0,1\}^{\bb T^d_N} \to \bb R$. Fix a typical realization $\omega \in \Omega$, and let $x \in \bb T_N^d$, define
\[
f(x,\eta)=f(x, \eta,\omega) =: f(T_{Nx}\omega, \tau_x\eta),
\]
where $\tau_x\eta$ is the shift of $\eta$ to $x$: $\tau_x \eta(y) = \eta(x+y)$. 

We say that $f$ is local if there exists $R>0$ such that $f(\omega,\eta)$ depends only on the values of $\eta(y)$ for $|y| \leq R$. In this case, we can consider $f$ as defined in all the spaces $\Omega \times \{0,1\}^{\bb T^d_N}$ for $N \geq R$.

 We say that $f$ is Lipschitz if there exists $c=c(\omega) >
0$ such that for all $x$, $|f(\omega, \eta)-f(\omega,\eta')| \leq c|\eta(x)-\eta'(x)|$ for any $\eta,\eta' \in \{0,1\}^{\bb T^d_N}$ such that $\eta(y) =\eta'(y)$ for
any $y \neq x$. If the constant $c$ can be chosen independently of $\omega$, we say that $f$ is uniformly Lipschitz.

\begin{theorem}{(Boltzmann-Gibbs principle)}
\label{th:bg}

For every $G\in \mc S_W(\bb T^d)$, every $t>0$ and every local, uniformly Lipschitz function $f: \Omega \times \{0,1\}^{\bb T^d_N} \to \bb R$,
\begin{equation} \label{expBG}
\lim_{N\rightarrow{\infty}}{E}_{\nu_{\rho}}\Big[\int_{0}^{t} \frac{1}{N^{d/2}} \sum_{x\in{\mathbb{T}_{N}^{d}}} G(x)V_{f}(x,\eta_s)ds
\Big]^{2}=0
\end{equation}
where
\begin{equation*}
V_{f}(x,\eta) =f(x,\eta)-E_{\nu_{\rho}}\big[f(x,\eta)\big]-\partial_{\rho}
E\Big[\int
f(x,\eta)d\nu_{\rho}(\eta)\Big]\big(\eta(x)-\rho\big).
\end{equation*}
Here, $E$ denotes the expectation with respect to $P$, the random environment.
\end{theorem}

Let $f: \Omega \times \{0,1\}^{\bb T^d_N} \to \bb R$ be a local, uniformly Lipschitz function and take  $f(x,\eta)=f(\theta_{Nx}\omega,\tau_x\eta)$. Fix a function $G\in{\mc S_W(\bb T^d)}$ and an integer $K$ that shall increase to $\infty$ after $N$. For each $N$, we subdivide $\mathbb{T}_N^d$ into non-overlapping boxes of linear size $K$. Denote them by $\{B_{j}, 1\leq{j}\leq{M}^{d}\}$, where $M=[\frac{2N}{K}]$. More precisely, 
$$B_j = y_j + \{1,\ldots,K\}^d,$$
where $y_{j}\in\bb T_N^d$, and $B_i\cap B_j=\emptyset$ if $i\neq j$. We assume that the points $y_j$ have the same relative position on the boxes.

Let $B_{0}$ be the set of
points that are not included in any $B_{j}$, then $|B_{0}|\leq{dKN^{d-1}}$. If we restrict the sum in the expression that appears
inside the integral in (\ref{expBG}) to the set $B_{0}$, then its $L^{2}_{\nu_\rho}(\bb T^d)$-norm clearly vanishes as
$N\rightarrow{+\infty}$.

Let $\Lambda_{s_{f}}$ be the smallest cube centered at the origin that contains the support of $f$ and define $s_f$ as the radius of
$\Lambda_{s_f}$. Denote by $B_{j}^{0}$ the interior of the box $B_{j}$, namely the sites $x$ in $B_{j}$ that are at a distance at least
$s_{f}+2$ from the boundary:
\begin{equation*}
B_{j}^{0}=\{x\in{B}_{j}, d(x,\mathbb{T}_{N}^{d}\setminus{B_{j}})>{s_{f}+2}\}.
\end{equation*}

Denote also by $B^{c}$ the set of points that are not included in any $B_{j}^{0}$. By construction, it is easy to see that
$|B^{c}|\leq{dN^{d}(\frac{c(f)}{K}+\frac{K}{N})}$, where $c(f)$ is a constant that depends on $f$. Recall that
$$V_{f}(x,\eta) =f(x,\eta)-E_{\nu_{\rho}}\big[f(x,\eta)\big]-\partial_{\rho}
E\Big[\int
f(x,\eta)d\nu_{\rho}(\eta)\Big]\big(\eta(x)-\rho\big)$$
 Thus, we have that for continuous $H:\bb T^d\to \bb R$,

\begin{multline*}
\frac{1}{N^{d/2}}\sum_{x\in{\mathbb{T}_{N}^d}}
    H(x)V_{f}(x,\eta_t)=
\frac{1}{N^{d/2}}\sum_{x\in{B}^{c}}H(x)V_{f}(x,\eta_t) +\\
    +\frac{1}{N^{d/2}}\sum_{j=1}^{M^d}\sum_{x\in{B}_{j}^{0}}\Big[H(x)-H(y_{j})\Big]V_{f}(x,\eta_t)
    +\frac{1}{N^{d/2}}\sum_{j=1}^{M^d}H(y_{j})\sum_{x\in{B_{j}^{0}}}V_{f}(x,\eta_t).
\end{multline*}
Note that we may take $H$ continuous, since the continuous functions are dense in $L^2(\bb T^d)$.
The first step is to prove that
\begin{equation*}
\lim_{K\rightarrow{\infty}}\lim_{N\rightarrow{\infty}}{E}_{\nu_{\rho}}
\Big[\int_{o}^{t}\frac{1}{N^{d/2}}\sum_{x\in{B}^{c}}H(x)V_{f}(x,\eta_t)ds\Big]^{2}=0.
\end{equation*}

Applying Cauchy-Schwartz inequality, since $\nu_{\rho}$ is an invariant product measure and since $V_{f}$ has mean zero with respect to the measure
$\nu_{\rho}$, the last expectation is bounded above by
\begin{equation*}
\frac{t^{2}}{N^{d}}\sum_{\substack{x,y\in{B^{c}}\\|x-y|\leq{2s_{f}}}}H(x)H(y) E_{\nu_{\rho}}\big[V_{f}(x,\eta)V_{f}(y,\eta)\big].
\end{equation*}

Since $V_{f}$ belongs to ${L}^{2}_{\nu_{\rho}}(\bb T^d)$ and
$|B^{c}|\leq{dN^{d}(\frac{c(f)}{K}+\frac{K}{N})}$, the last expression
vanishes by taking first $N\rightarrow{+\infty}$ and then
$K\rightarrow{+\infty}$.

From the continuity of $H$, and applying similar arguments, one may show that
\begin{equation*}
\lim_{N\rightarrow{\infty}}\mathbb{E}_{\nu_{\rho}}\Big[\int_{0}^{t}\frac{1}{N^{d/2}}\sum_{j=1}^{M^d}\sum_{x\in{B}_{j}^{0}}
\big[H(x)-H(y_{j})\big]V_{f}(x,\eta_t)ds\Big]^{2}=0.
\end{equation*}

In order to conclude the proof it remains to be shown that
\begin{equation*}
\lim_{K\rightarrow{\infty}}\lim_{N\rightarrow{\infty}} {E}_{\nu_{\rho}}\Big[\int_{0}^{t}\frac{1}{N^{d/2}}\sum_{j=1}^{M^d}H(y_{j})
\sum_{x\in{B_{j}^{0}}}V_{f}(x,\eta_t)ds\Big]^{2}=0.
\end{equation*}

To this end, let $\tilde{L}_{N}$ be the generator of the exclusion process without the random environment, and without the conductances (that is, taking $a(\omega) \equiv 1$, and $W_j(x_j)=x_j$, for $j=1,\ldots,d$, in (\ref{g4})),
and also without the diffusive scaling $N^2$
$$\tilde{L}_N g(\eta) = \sum_{j=1}^d\sum_{x\in \bb T^d_N}c_{x,x+e_j}(\eta)\big[g(\eta^{x,x+e_j})-g(\eta)\big],$$
for cylindric functions $g$ on the configuration space $\{0,1\}^{\bb T_N^d}$.

 For each $j=0,..,M^d$ denote by $\zeta_{j}$ the configuration $\{\eta(x), x \in{B_{j}}\}$ and by $\tilde{L}_{B_{j}}$ the restriction of the generator $\tilde{L}_{N}$ to the box $B_{j}$, namely:
\begin{equation*}
\tilde{L}_{B_{j}} h(\eta) = \sum_{\substack{ x,y \in B_j \\ |x-y|= 1/N}}c_{x,x+e_j}(\eta)\big[h(\eta^{x,x+e_j}-h(\eta)\big].
\end{equation*}

We would like to emphasize that we introduced the generator $\tilde{L}_N$ because it is translation invariant. 

Now we introduce some notation. Fix a local function $h: \Omega \times \{0,1\}^{\bb T^d_N} \to \bb R$, measurable with respect to $\sigma(\eta(x),x\in{B_{1}})$, such that $E[\int h(\omega,\eta)^2 d\nu_\rho] < \infty$ and let $h_{j}$ be the translation of $h$ by $y_j -y_0$:
$h_j(x,\eta) = h(\theta_{(y_j-y_0)N} \omega, \tau_{y_j-y_0} \eta)$. Denote by $L^2(\nu_\rho\otimes P)$ the set of such functions. Consider
\begin{equation*}
V_{H,h}^{N}(\eta)=\frac{1}{N^{d/2}}\sum_{j=1}^{M^d}H(y_{j}) \tilde{L}_{B_{j}}h_{j}(\zeta_{j}).
\end{equation*}

By proposition A 1.6.1 of \cite{kl} and the ellipticity assumption, it is not hard to show that
\begin{equation}\label{desigdual}
E_{\nu_\rho}\left[\int_0^t V_{H,h}^N(\eta_s) ds\right] \leq 20 \theta t\|V_{H,h}^N\|_{-1}^2,
\end{equation}
where $\| \cdot\|_{-1}$ is given  by
$$\| V\|_{-1}^2 = \sup_{F\in L^2(\nu_\rho)} \left\{2\int V(\eta)F(\eta)d\nu_\rho - \< F,L_NF\>_\rho \right\},$$
and $\<\cdot,\cdot\>_\rho$ denotes the inner product in $L^2(\nu_\rho)$. Moreover, $\eta_s$ is the evolution of the process with generator (in the box $B_j$):
$$L_{W,B_j} h(\eta) = \sum_{i=1}^d \sum_{x\in B_j} c_{x,x+e_j}(\eta) \frac{N a_i(x)}{W(x+e_i)-W(x)} [h(\eta^{x,x+e_i})-h(\eta)].$$

Since $L_{W,B_j}$ is a decomposition of $L_N$, we have that
$$\sum_{j=1}^{M^d} \< h,-L_{W,B_j} h\>_\rho \leq \<h,-L_N h \>_\rho.$$

Furthermore,
$$\<f, -\tilde{L}_{B_j} h \> \leq \max_{1\leq k\leq d} \frac{\{W_k(1)-W_k(0)\}}{N} \theta \<h,-L_{W,B_j}h\>_\rho.$$

Using the Cauchy-Schwartz inequality, we have, for each $j$,
$$\<\tilde{L}_{B_j}h_j, F \>_\rho \leq \frac{1}{2\gamma_j}\<-\tilde{L}_{B_j}h_j,h_j \>_\rho + \frac{\gamma_j}{2}\<F,-\tilde{L}_{B_j}F \>_\rho,$$
where $\gamma_j$ is a positive constant.

Therefore,

\begin{equation}\label{vgamma}
2\int V_{H,h}^N(\eta) F(\eta) d\nu_\rho \leq \frac{2}{N^{d/2}}\sum_{j=1}^{M^d} H(y_j) \left[\frac{1}{2\gamma_j}\<-\tilde{L}_{B_j}h_j,h_j \>_\rho+\frac{\gamma_j}{2} \<F,-\tilde{L}_{B_j}F \>_\rho \right].
\end{equation}

Choosing
$$\gamma_j = \frac{N^{1+d/2}}{\theta \max_{1\leq k \leq d}\{W_k(1)-W_k(0)\} |H(y_j)|}$$
Observe that the generator $L_N$ is already speeded up by the factor $N^2$.
We, thus, obtain
$$\frac{2}{N^{d/2}} \sum_{j=1}^{M^d} H(y_j) \frac{\gamma_j}{2} \<F,-\tilde{L}_{B_j}F \>_\rho \leq \<F,-\tilde{L}_NF \>_\rho,$$
where, from \eqref{desigdual} and \eqref{vgamma}, we obtain that the expectation in equation \eqref{desigdual} is bounded above by

$$\frac{20\theta t}{N^{d/2}} \sum_{j=1}^{M^d} \frac{|H(y_j)|}{\gamma_j} \<-\tilde{L}_{B_j}h_j,h_j \>_\rho,$$
which in turn is less than or equal to
$$\frac{20t\|H\|_{\infty} M^d\theta^2}{N^{d+1}\max_{1\leq k\leq d}\{W_k(1)-W_k(0)\}} \sum_{j=1}^{M^d}\frac{1}{M^d}\<-\tilde{L}_{B_j}h_j,h_j \>_\rho.$$
By Birkhoff's ergodic theorem, the sum in the previous expression converges to a finite value as $N \to \infty$. Therefore, this whole expression
vanishes as $N\rightarrow{\infty}$. To conclude the proof of the theorem we need to show that
\begin{multline*}
\lim_{K\rightarrow{\infty}}\inf_{h\in{L^{2}(\nu_{\rho}\otimes P)}} \\
\lim_{N\rightarrow{\infty}}{E}_{\nu_{\rho}}\Big[\int_{0}^{t}\frac{1}{N^{d/2}}
\sum_{j=1}^{M^d}H(y_{j})\Big\{\sum_{x\in{B_{j}^{0}}}V_{f}(x,\eta_s)-\tilde{L}_{B_{j}}h_{j}(\zeta_{j}(s))\Big\}\Big]^{2}=0.
\end{multline*}

By Cauchy-Schwartz inequality the expectation in the previous expression is bounded by
\begin{equation*}
\frac{t^{2}}{N^{d}}\sum_{j=1}^{M^d}||H||_\infty^{2} E_{\nu_{\rho}}\Big(\sum_{x\in{B_{j}^{0}}}V_{f}(x,\eta)- \tilde{L}_{B_{j}}h_{j}(\zeta_{j})\Big)^{2}
\end{equation*}
because the measure $\nu_{\rho}$ is invariant under the dynamics and also translation invariant and the supports of $V_{f}(x,\eta)-
\tilde{L}_{B_{i}}h_{i}(\zeta_{i})$ and $V_{f}(y,\eta)- \tilde{L}_{B_{j}}h_{j}(\zeta_{j})$ are disjoint for $x\in{B_{i}^{0}}$ and $y\in{B_{j}^{0}}$,
 with $i\neq{j}$.

By the ergodic theorem, as $N\rightarrow{\infty}$ this expression converges to
\begin{equation}
\frac{t^{2}}{K^{d}}||H||_\infty^2 E\Big[\int \Big(\sum_{x\in{B_{1}^{0}}}V_{f}(x,\eta)- \tilde{L}_{B_{1}}h(\omega,\eta)\Big)^{2} d\nu_\rho\Big].
\label{eq:bg2}
\end{equation}

So it remains to be shown that

\begin{equation*}
\lim_{K\rightarrow{\infty}} \frac{t^{2}}{K^{d}} ||H||_\infty^2
    \inf_{h\in{L^{2}(\nu_{\rho} \otimes P)}}
    E\Big[ \int \Big(\sum_{x\in{B_{1}^{0}}}V_{f}(x,\eta)
    - \tilde{L}_{B_{1}}h(\omega, \eta)\Big)^{2} d \nu_\rho \Big]=0.
\end{equation*}

Denote by ${R}(\tilde{L}_{B_{1}})$ the range of the generator $\tilde{L}_{B_{1}}$ in $L^{2}(\nu_{\rho} \otimes P)$ and by
${R}(\tilde{L}_{B_{1}})^{\perp}$ the space orthogonal to ${R}(\tilde{L}_{B_{1}})$. The infimum  of (\ref{eq:bg2}) over all
$h\in{L^{2}(\nu_{\rho} \otimes P)}$ is equal to the projection of $\sum_{x\in{B_{1}^{0}}}V_{f}(x,\eta)$ into ${R}(\tilde{L}_{B_{1}})^{\perp}$.

The set ${R}(\tilde{L}_{B_{1}})^{\perp}$ is the space of functions that depends on $\eta$ only through the total
number of particles on the box $B_1$. So, the previous expression is equal to
\begin{equation}
\label{ec9} \lim_{K \to \infty} \frac{t^{2}||H||_\infty^2}{K^{d}} E\Big[\int
\Big(E_{\nu_{\rho}}\Big[\sum_{x\in{B_{1}^{0}}}V_{f}(x,\eta)\Big|\eta^{B_{1}}\Big]\Big)^{2} d\nu_\rho \Big]
\end{equation}
where $\eta^{B_{1}}=K^{-d}\sum_{x\in{B_{1}}}\eta(x)$.

Let us call this last expression $\mc I_0$. Define $\psi(x,\rho) = E_{\nu_\rho}[ f(\theta_x \omega)]$. Notice that $V_f(x,\eta)= f(x,\eta) -
\psi(x,\rho) - E[\partial_\rho\psi(x,\rho)]\big(\eta(x)-\rho\big)$, since in the last term the partial derivative with respect to $\rho$ commutes with the expectation with respect to the
random environment. In order to estimate the expression (\ref{ec9}), we use the elementary inequality $(x+y)^2\leq{2x^2+2y^2}$. Therefore, we obtain $\mc I_0 \leq 4(\mc I_1+ \mc I_2 +\mc I_3)$, where
\[
\mc I_1 = \frac{1}{K^d} E\Big[ \int \Big(\sum_{x \in B_1^0} E_{\nu_\rho} \big[f(x,\eta)|\eta^{B_1}\big] - \psi(x,\eta^{B_1})\Big)^2 d\nu_\rho\Big],
\]
\[
\mc I_2 =  \frac{1}{K^d} E\Big[ \int \Big(\sum_{x \in B_1^0} \psi(x,\eta^{B_1})
-\psi(x,\rho) - \partial_\rho\psi(x,\rho)[\eta^{B_1}-\rho] \Big)^2 d\nu_\rho \Big],
\]
\[
\mc I_3 = \frac{1}{K^d} E\Big[ E_{\nu_\rho}\Big[ \Big(\sum_{x \in B_1^0}\big(\partial_\rho\psi(x,\rho) - E[\partial_\rho\psi(x,\rho)]\big)\big[\eta^{B_1} - \rho\big] \Big)^2\Big]\Big].
\]

Recall the equivalence of ensembles (see Lemma A.2.2.2 in \cite{kl}):

\begin{lemma}
\label{eqens} Let $h: \{0,1\}^{\bb T_N^d} \to \bb R$ be a local uniformly Lipschitz function. Then, there exists a constant $C$
that depends on $h$ only through its support and its Lipschitz constant, such that
\[
\Big| E_{\nu_\rho} [ h(\eta)|\eta^N] - E_{\nu_{\eta^N}}[h(\eta)]\Big| \leq \frac{C}{N^d}.
\]
\end{lemma}

Applying Lemma \ref{eqens}, we get
\[
\frac{1}{K^d} E\Big[ \int \Big(\sum_{x \in B_1^0} E_{\nu_\rho} \big[f(x,\eta)|\eta^{B_1}\big] -
\psi(x,\eta^{B_1})\Big)^2 d\nu_\rho\Big] \leq \frac{C}{K^d},
\]
which vanishes as $K \to \infty$.

 Using a Taylor expansion for
$\psi(x,\rho)$, we obtain that
\[
\frac{1}{K^d} E\Big[ \int \Big(\sum_{x \in B_1^0} \psi(x,\eta^{B_1})
-\psi(x,\rho) - \partial_\rho\psi(x,\rho)[\eta^{B_1}-\rho] \Big)^2 d\nu_\rho
\Big] \leq \frac{C}{K^d},
\]
and also goes to 0 as $K \to \infty$.

Finally, we see that
\[
\mc I_3 = E_{\nu_\rho}\big[(\eta(0) -\rho)^2\big]\cdot E\Big[\Big( \frac{1}{K^d} \sum_{x\in B_1^0} (\partial_\rho\psi(x,\rho) - E[\partial_\rho\psi(x,\rho)]\Big)^2\Big]
\]
and it goes to 0 as $K \to \infty$ by the $L^2$-ergodic theorem. This concludes the proof of Theorem \ref{th:bg}.

\begin{appendix}
\section{Stochastic differential equations on nuclear spaces }\label{a}

\subsection{Countably Hilbert nuclear spaces}

In this subsection we introduce countably Hilbert nuclear spaces which will be the natural environment for the study of the stochastic evolution equations obtained from the martingale problem. Our strategy here is to recall some basic definitions of these spaces. To this end, we follow the ideas of Kallianpur and Perez-Abreu \cite{k,kp} and Gel'fand and Vilenkin \cite{g}.

Let $\Phi$ be a (real) linear space, whose topology $\tau$ is given by an increasing sequence $\|\cdot\|_r, \;\; r\in\bb N$, of Hilbertian norms. Let $\Phi_r$ be the completion of $\Phi$ with respect to $\|\cdot\|_r$, and assume that
\begin{equation*}
\Phi_\infty = \bigcap_{r=1}^\infty\Phi_r.
\end{equation*}

Then $(\Phi_\infty, \tau)$ is a Fréchet space with respect to the metric
\begin{equation}\label{metrica}
\rho(f,g) = \sum^\infty_{r=1}2^{-r}\frac{\|f-g\|_r}{1+\|f-g\|_r},
\end{equation}
and$(\Phi_\infty, \rho)$ is called  a Countably Hilbert space. Since for $n\le m$
\begin{equation}\label{normdesig}
\|f\|_n\le\|f\|_m\qquad\text{for all  } f\in \Phi,
\end{equation}
we have,
$$\Phi_m\subset\Phi_n\;\;\text{for all}\;\;m\ge n.$$

A countably Hilbert space $\Phi_\infty$ is called \textit{Nuclear} if for each $n\ge0$, there exists $m>n$ such that the canonical injection $\pi_{m,n}:\Phi_m\to \Phi_n$ is Hilbert-Schmidt, i.e., if $\{f_j\}_{j\ge1}$ is a complete orthonormal system in $\Phi_m$  we have
$$\sum_{j=1}^\infty\|f_j\|^2_n<\infty.$$

We now characterize the topologic dual $\Phi_\infty'$ of the countably Hilbert nuclear space $\Phi_\infty$ in terms of the topologic dual of the auxiliary spaces $\Phi_n$.

Let $\Phi'_n$ be the dual (Hilbert) space of  $\Phi_n$, and for $\phi\in \Phi'_n$ let
$$\|\phi\|_{-n} = \sup_{\|f\|_n\le1}|\phi[f]|,$$
which, by equation \eqref{normdesig}, implies that
$$\Phi'_n\subset\Phi'_m\;\;\text{for all}\;\;m\ge n.$$

Let $\Phi'_\infty$ be the topologic dual of $\Phi_\infty$ with respect to the strong topology, which is given by the complete system of neighborhoods of zero given by sets of the form, $\{\phi\in \Phi'_\infty:\;\|\phi\|_B<\epsilon\}$, where
$\|\phi\|_B = \sup\{|\phi[f]|:\; f\in B\}$ and $B$ is a bounded set in $\Phi_\infty$.
So,
$$\Phi'_\infty = \bigcup_{r=1}^\infty\Phi'_r.$$

\subsection{Stochastic differential equations}
The aim of this subsection is to recall some results about existence and uniqueness of stochastic evolution equations in nuclear spaces.

We denote by $\mc L(\Phi_\infty,\Phi_\infty)$ (resp. $\mc L(\Phi'_\infty,\Phi'_\infty)$) the class of continuous linear operators from $\Phi_\infty$ to $\Phi_\infty$ (resp.$\Phi'_\infty$ to $\Phi'_\infty$).

A family $\{S(t):t\ge0\}$ of the linear operators on $\Phi_\infty$ is said to be a $C_{0,1}$-semigroup if the following three conditions are satisfied:
\begin{itemize}
\item $S(t_1)S(t_2)  = S(t_1+t_2)$ for all $t_1,t_2\ge0,\;\;S(0) = I$.
\item The map $t\to S(t)f$ is $\Phi_\infty$-continuous for each $f\in \Phi_\infty$.
\item For each $q\ge0$ there exist numbers $M_q>0,\sigma_q>0$ and $p\ge q$ such that
\begin{equation*}
\|S(t)f\|_q\le M_q \; e^{\sigma_q t}\|f\|_p\;\;\;\text{for all  }\;\; f\in\Phi_\infty, \;\;t>0.
\end{equation*}
\end{itemize}

Let $A$ in $\mc L(\Phi_\infty,\Phi_\infty)$ be infinitesimal generator of the semigroup $\{S(t):t\ge0\}$ in $\mc L(\Phi_\infty,\Phi_\infty)$.
The relations 
\begin{eqnarray*}
\phi[S(t)f]&:=& (S'(t)\phi)[f]\;\;\text{for all}\;\;t\ge0,\;f\in\Phi_\infty\;\; \text{and}\;\;\phi\in\Phi'_\infty;\\
\phi[A f]&:=& (A' \phi)[f]\;\;\text{for all}\;\;f\in\Phi_\infty\;\; \text{and}\;\;\phi\in\Phi'_\infty;
\end{eqnarray*}
define the infinitesimal generator $A'$ in $\mc L(\Phi'_\infty,\Phi'_\infty)$ of the semigroup $\{S'(t):t\ge0\}$ in $\mc L(\Phi'_\infty,\Phi'_\infty)$.

Let $(\Sigma,\mc U, P)$  be a complete probability space with a right continuous filtration $(\mc U_t)_{t\ge0}$, $\mc U_0$ containing all the $P$-null sets of $\mc U$, and $M=(M_t)_{t\ge0}$ be a $\Phi'_\infty$-valued martingale with respect to $\mc U_t$, i.e., for each $f\in \Phi_\infty$, $(M_t[f], \mc U_t)$ is a real-valued martingale. 
We are interested in results of existence and uniqueness of the following $\Phi'_\infty$-valued stochastic evolution equation:
\begin{equation}\label{see}
\begin{array}{ccc}
d\xi_t &=& A'\xi_t dt + d M_t, \;\;\;\;t>0,\\
\xi_0&=&\gamma ,
\end{array}
\end{equation}
where $\gamma$ is a $\Phi'_\infty$-valued random variable, and $A$ is the infinitesimal generator of a $C_{0,1}$-semigroup on $\Phi_\infty$. 

We say that $\xi = (\xi_t)_{t\ge0}$ is a $\Phi'_\infty$-solution of the stochastic evolution equation \eqref{see} if the following conditions are satisfied:
\begin{itemize}
\item $\xi_t$ is $\Phi'_\infty$-valued, progressively measurable, and $\mc U_t$-adapted;
\item the following integral identity holds:
$$\xi_t[f] = \gamma[f] + \int_0^t\xi_s[Af]ds + M_t[f],$$
for all $f\in \Phi_\infty,\;t\ge0$ a.s..
\end{itemize} 

It is proved in \cite{kp} the following result on existence and uniqueness of solutions of the stochastic differential equation \eqref{see}:
\begin{proposition}\label{eu-spde}
Assume the conditions below:
\begin{enumerate}
\item $\gamma$ is a $\Phi'_\infty$-valued $\mc U_0$-measurable random element such that, for some $r_0>0,\;E|\gamma|^2_{-r_0}<\infty$;
\item $M=(M_t)_{t\ge0}$ is a $\Phi'_\infty$-valued martingale such that $M_0=0$ and, for each $t\ge0$ and $f\in\Phi,\;E(M_t[f])^2<\infty$;
\item $A$ is a continuous linear operator on $\Phi_\infty$, and is the infinitesimal generator of a $C_{0,1}$-semigroup $\{S(t):\;t\ge0\}$ on $\Phi_\infty$.
\end{enumerate}
Then the $\Phi'_\infty$-valued homogeneous stochastic evolution equation \eqref{see} has a unique solution $\xi = (\xi_t)_{t\ge0}$ given explicitly by the ``evolution solution":
\begin{equation*}
\xi_t = S'(t)\gamma + \int^t_0S'(t-s)dM_s.
\end{equation*}
\end{proposition}
\begin{remark}\label{remark}
The condition 2 of Proposition \ref{eu-spde} can be obtained if $E(M_t[f])^2 = t Q(f,f),$ where $f\in\Phi_\infty$, and $Q(\cdot,\cdot)$ is a positive definite continuous bilinear form on $\Phi_\infty\times\Phi_\infty$.
\end{remark}

We now state a proposition that gives a sufficient condition for the solution $\xi_t$ of the equation \eqref{see} be a Gaussian process.
\begin{proposition}\label{gaussian}
Assume $\gamma$ is a $\Phi'_\infty$-valued Gaussian element independent of the $\Phi'_\infty$-valued Gaussian martingale with independent increments $M_t$. Then, the solution $\xi=(\xi_t)$ of \eqref{see} is a $\Phi'_\infty$-valued Gaussian process.
\end{proposition}

\end{appendix}
\section*{Acknowledgements}
We would like to thank Milton Jara for valuable discussions on nuclear spaces.

\end{document}